\documentclass{amsart}

\usepackage[alphabetic,initials,nobysame]{amsrefs}
\usepackage{amssymb,mathtools,mathrsfs,comment,hyperref}
\usepackage[shortlabels]{enumitem}
\usepackage{color}
\usepackage{caption}
\usepackage[percent]{overpic}
\usepackage{tikz}

\DeclareFontFamily{U}{mathx}{\hyphenchar\font45}
\DeclareFontShape{U}{mathx}{m}{n}{
      <5> <6> <7> <8> <9> <10>
      <10.95> <12> <14.4> <17.28> <20.74> <24.88>
      mathx10
      }{}
\DeclareSymbolFont{mathx}{U}{mathx}{m}{n}
\DeclareFontSubstitution{U}{mathx}{m}{n}
\DeclareMathAccent{\widecheck}{0}{mathx}{"71}

\theoremstyle{plain}
\newtheorem{theorem}{Theorem}
\newtheorem{lemma}[theorem]{Lemma}

\newtheorem{corollary}[theorem]{Corollary}

\theoremstyle{definition}

\newtheorem{innercustomthm}{Case}
\newenvironment{case}[1]
  {\renewcommand\theinnercustomthm{#1}\innercustomthm}
  {\endinnercustomthm}

\theoremstyle{remark}
\newtheorem{remark}[theorem]{Remark}
\newtheorem{question}[theorem]{Question}

\numberwithin{equation}{section}
\numberwithin{theorem}{section}
\numberwithin{conjecture}{section}

\newcommand{\br}{\overline}
\newcommand{\R}{\mathbb R}

\newcommand{\C}{\mathbb C}
\newcommand{\D}{\mathbb D}

\newcommand{\UHP}{\mathbb H}

\DeclareMathOperator{\dist}{{\mathrm{dist}}}

\DeclareMathOperator{\im}{{\mathrm{Im}}}

\DeclareMathOperator{\loc}{\mathrm{loc}}

\begin{document}
\title[Extension of boundary maps to mappings of finite distortion]{Extension of boundary homeomorphisms to mappings of finite distortion}

\author{Christina Karafyllia}
\author{Dimitrios Ntalampekos}
\address{Institute for Mathematical Sciences, Stony Brook University, Stony Brook, NY 11794, USA.}

\thanks{The second author is partially supported by NSF Grant DMS-2000096.}
\email{christina.karafyllia@stonybrook.edu} 
\email[Corresponding author]{dimitrios.ntalampekos@stonybrook.edu}

\date{\today}
\keywords{quasiconformal mapping, mapping of finite distortion, Beurling--Ahlfors extension, distortion, dilatation, exponentially integrable distortion, David homeomorphism, $BMO$-quasiconformal mapping}
\subjclass[2020]{Primary 30C62, 30C65; Secondary 37F31, 46E35.}

\begin{abstract}
We provide sufficient conditions so that a homeomorphism of the real line or of the circle admits an extension to a mapping of finite distortion in the upper half-plane or the disk, respectively. Moreover, we can ensure that the quasiconformal dilatation of the extension satisfies certain integrability conditions, such as $p$-integrability or exponential integrability. Mappings satisfying the latter integrability condition are also known as David homeomorphisms. Our extension operator is the same as the one used by Beurling and Ahlfors in their celebrated work. We prove an optimal bound for the quasiconformal dilatation of the Beurling--Ahlfors extension of a homeomorphism of the real line, in terms of its symmetric distortion function. More specifically, the quasiconformal dilatation is bounded above by an average of the symmetric distortion function and below by the symmetric distortion function itself. As a consequence, the quasiconformal dilatation of the Beurling--Ahlfors extension of a homeomorphism of the real line is (sub)exponentially integrable, is $p$-integrable, or has a $BMO$ majorant  if and only if the symmetric distortion is (sub)exponentially integrable, is $p$-integrable, or has a $BMO$ majorant, respectively.  These theorems are all new and  reconcile several sufficient extension conditions that have been established in the past.

\end{abstract}

\maketitle

\section{Introduction}

The goal of this work is to provide extension theorems for homeomorphisms of the real line or of the circle whose regularity is beyond the quasisymmetric class (defined below). While quasisymmetric homeomorphisms are suitable for studying self-similar sets, or sets with uniform geometry, they are not sufficient for the study of fractals with non-uniform geometry. Such fractals appear often in the field of Complex Dynamics as Julia sets of non-hyperbolic rational maps. Hence, extension theorems for homeomorphisms beyond the quasisymmetric class provide valuable tools for studying non-hyperbolic dynamical systems. 

Several recent works in the field are based on extensions of homeomorphisms of the circle to David homeomorphisms of the disk (defined below). More specifically, David extensions are useful for turning a hyperbolic dynamical system into a parabolic one, an observation that was originally made by Ha\"{\i}ssinsky. Thus, parabolic systems can be studied in terms of hyperbolic systems, which are much better understood. In some instances these extensions have been constructed ``by hand''; see \cite{Haissinsky:parabolic, PetersenZakeri:Siegel} and also \cite[Chapter 9]{BrannerFagella:surgery}. Later Zakeri \cite{Zakeri:boundary} studied systematically extension problems and provided a useful criterion for David extensions, based on the work of J.\ Chen, Z.\ Chen, and He \cite{ChenChenHe:boundary}. We will discuss these results later in detail. We also cite the recent works \cite{LodgeLyubichMerenkovMukherjee:gaskets, LyubichMerenkovMukherjeeNtalampekos:David}, where David extensions of circle homeomorphisms have been used successfully in the study of geometrically finite rational maps and Kleinian groups.  Our results in this paper are stronger than the existing extension theory. Hence, we expect that they will provide useful tools for further developments in Complex Dynamics and they will broaden the understanding of mappings of finite distortion (defined below). 

Let $h\colon \R \to \R$ be an increasing homeomorphism. For $x\in \R$ and $t>0$ we define the \textit{symmetric distortion function}
$$\rho_h(x,t)= \max \left\{ \frac{|h(x+t)-h(x)|}{|h(x)-h(x-t)|}, \frac{|h(x)-h(x-t)|}{|h(x+t)-h(x)|}\right\}.$$
The symmetric distortion function measures how far the homeomorphism $h$ is from mapping adjacent intervals of equal length to adjacent intervals of equal length. If there exists $\varrho>0$ such that $\rho_h(x,t)\leq \varrho$ for all $x\in \R$, $t>0$, then $h$ is called \textit{quasisymmetric}. Beurling and Ahlfors proved in \cite{BeurlingAhlfors:extension} that 
if $h$ is quasisymmetric, then there exists a \textit{quasiconformal} extension of $h$ to the upper half plane. Recall that a homeomorphism $H\colon U\to V$ between two open sets $U,V\subset \R^2$ is quasiconformal if $H$ is orientation-preserving, $H$ lies in the Sobolev space  $W^{1,1}_{\loc}(U)$, the Jacobian $J_H$ lies in $L^1_{\loc}(U)$, and the \textit{quasiconformal dilatation}
\begin{align*}
K_H(x,y)= \inf\{ K\geq 1 : \|DH(x,y)\|^2\leq K J_H(x,y)  \}
\end{align*}
of $H$ lies in $L^\infty$, where $\|DH\|$ denotes the operator norm of the differential matrix of $H$. 

If one relaxes the assumption that $K_H\in L^\infty$ to merely $K_H<\infty$ a.e., then we say that $H$ is a mapping of \textit{finite distortion}. See \cite{Koskela:survey} for an enlightening survey and \cite{HenclKoskela:finitedistortion} for a treatise on the general theory of these mappings. Among mappings of finite distortion, of particular interest are the mappings of \textit{exponentially integrable distortion}, or else \textit{David homeomorphisms}, because of their increased regularity and of the fact that they provide a substitute for quasiconformal maps in many cases when the use of the latter is not possible, such as in the framework of Complex Dynamics mentioned above. These maps were introduced by David in \cite{David:maps} and their defining condition is that 
\begin{align*}
\int_U e^{pK_H(x,y)} \, d\sigma(x,y) <\infty
\end{align*}
for some $p>0$, where $\sigma$ denotes the spherical measure on the Riemann sphere $\widehat{\C}$ and $U\subset \widehat{\C}$ is an open set. We prove the following result, which provides extensions of boundary homeomorphisms to mappings that have exponentially integrable distortion. We also refer to these extensions as \textit{David extensions}. We denote by $\UHP$ the upper half-plane and by $\D$ the unit disk in the plane.

\begin{theorem}\label{theorem:david_line}
Let $h\colon \R\to\R$ be an increasing homeomorphism such that 
\begin{align*}
\int_{\UHP} e^{q \rho_h(x,y)} \, d\sigma(x,y) <\infty
\end{align*}
for some $q>0$. Then there exists an extension of $h$ to a homeomorphism of $\UHP$ that has exponentially integrable distortion. 
\end{theorem}

An analogous theorem can also be formulated for homeomorphisms of the circle. For $a,b\in S^1$ we denote by $\ell(a,b)$ the length of the arc of the circle that connects $a$ to $b$ in the positive orientation. If $h\colon S^1\to S^1$ is an orientation-preserving homeomorphism, we define the \textit{circular} symmetric distortion function
\begin{align*}
\rho_h^c(\theta,t)= \max \left\{ \frac{\ell(h(e^{i\theta}),h(e^{i(\theta+t)}))}{\ell(h(e^{i(\theta-t)}),h(e^{i\theta}))}, \frac{\ell(h(e^{i(\theta-t)}),h(e^{i\theta}))}{\ell(h(e^{i\theta}),h(e^{i(\theta+t)}))} \right\}
\end{align*}
for $\theta\in [0,2\pi]$ and $t\in (0,\pi/2)$. One could alternatively use Euclidean distances in the definition of the circular symmetric distortion function and obtain a quantity that is comparable to $\rho_h^c$ for all small $t$.

\begin{theorem}\label{theorem:david_circle}
Let $h\colon S^1\to S^1$ be an orientation-preserving homeomorphism such that
\begin{align*}
\int_0^{2\pi} \int_0^{\pi/2} e^{q \rho_h^c(\theta,t)} \, dtd\theta <\infty
\end{align*}
for some $q>0$. Then there exists an extension of $h$ to a homeomorphism of $\D$ that has exponentially integrable distortion.
\end{theorem}

\begin{corollary}\label{corollary:david_circle}
Let $h\colon S^1\to S^1$ be an orientation-preserving homeomorphism and suppose that there exists a non-negative function $g\in L^1([0,2\pi])$ such that
\begin{align*}
\rho_h^c(\theta,t) = O\left( \log \left(\frac{1+g(\theta)}{t} \right)\right)
\end{align*}
as $t\to 0$. Then there exists an extension of $h$ to a homeomorphism of $\D$ that has exponentially integrable distortion.
\end{corollary}

This extension result can be applied in order to turn hyperbolic dynamical systems into parabolic ones with global David homeomorphisms of the sphere. In fact, the main result of \cite{LyubichMerenkovMukherjeeNtalampekos:David} relies on a weaker version of this corollary from \cite{Zakeri:boundary}. Theorem \ref{theorem:david_circle} follows from the following more general result.

\begin{theorem}\label{theorem:convex_circle}
There exists a uniform constant $C_0>0$ such that the following holds. Let $\Phi\colon [0,\infty)\to [0,\infty)$ be an increasing convex function and suppose that $h\colon S^1\to S^1$ is an orientation-preserving homeomorphism such that
\begin{align*}
\int_0^{2\pi} \int_0^{\pi/2} \Phi({q \rho_h^c(\theta,t)}) \, dtd\theta <\infty
\end{align*}
for some $q>0$. Then there exists an extension of $h$ to a homeomorphism $H$ of $\D$ that has finite distortion and 
\begin{align*}
\int_{\D} \Phi( qC_0^{-1}K_H(x,y)) \, dxdy<\infty.
\end{align*}
\end{theorem}

The constant $C_0$ is the same as the constant of Theorem \ref{theorem:main} below. We prove Theorem \ref{theorem:convex_circle} in Section \ref{section:consequences}. For $\Phi(x)=e^x$, this theorem implies Theorem \ref{theorem:david_circle}. Moreover, if $\Phi(x)=x^q$, $q\geq 1$, then we obtain extensions having \textit{$q$-integrable distortion} and if $\Phi(x)=e^{x/\log(e+x)}$ then we obtain extensions of \textit{subexponentially integrable distortion}. The latter class of mappings is slightly weaker than David homeomorphisms, but their general theory has no essential differences. Moreover, the condition of subexponentially integrable distortion is very close to the optimal sufficient condition for obtaining solutions to the \textit{Beltrami equation}; see \cite[Section 20.5, p.~570]{AstalaIwaniecMartin:quasiconformal} for more background. We remark that Theorem \ref{theorem:convex_circle} generalizes the result of Zakeri \cite[Theorem B]{Zakeri:boundary}, which uses $\sup_{t>0} \rho_h^c(x,t)$ in place of $\rho_h^c$.

We do not know whether the assumption of Theorem \ref{theorem:convex_circle} is also necessary for extensions that have $q$-integrable or exponentially integrable distortion. In fact, so far there exists a necessary and sufficient condition only for mappings of $1$-integrable distortion, due to Astala, Iwaniec, Martin, and Onninen \cite[Theorem 11.1]{AstalaEtAl:ExtremalFiniteDistortion}. Namely, a homeomorphism $h\colon S^1\to S^1$ extends to a homeomorphism $H$ of the disk with $K_H\in L^1(\D)$ if and only if 
\begin{align*}
\int_{0}^{2\pi}\int_0^{2\pi}  \left |\log| h(e^{i\theta})-h(e^{i\phi})|  \right| \, d\theta d\phi<\infty.
\end{align*}
We pose the following question.
\begin{question}\label{question:necessary}
Is the sufficient condition
\begin{align*}
\int_0^{2\pi} \int_0^{\pi/2} e^{q \rho_h^c(\theta,t)} \, dtd\theta <\infty \quad \left(\textrm{resp.}\quad \int_0^{2\pi} \int_0^{\pi/2} \left(\rho_h^c(\theta,t)\right)^q \, dtd\theta <\infty \right)
\end{align*}
also necessary for obtaining an extension of exponentially integrable distortion (resp.\ $q$-integrable distortion, $q\geq 1$)?
\end{question}
A positive answer to the question would lead to great progress towards understanding mappings of exponentially integrable distortion and their boundary behavior. It is pointed out in \cite[pp.~248--249]{Zakeri:boundary} that the only known necessary condition for an extension of exponentially integrable distortion is $\rho_h^c(\theta,t) =O(t^{-\alpha})$ as $t\to0$ for some $\alpha>0$, which is much weaker than the condition in question. 

Next, we discuss the main theorem that leads to all the mentioned results. Let $h\colon \R \to \R$ be an increasing homeomorphism. Beurling and Ahlfors constructed in \cite{BeurlingAhlfors:extension} an operator that extends $h$ to a $C^1$-diffeomorphism of the upper half-plane. We denote by $K_h$ the quasiconformal dilatation of the extension. They showed that if  $\rho_h(x,t)\leq \varrho$ for some $\varrho>0$, then the extension of $h$ is quasiconformal and moreover $K_h\leq \varrho^2$. This bound was later improved by Reed  \cite{Reed:BeurlingAhlfors} to $K_h\leq 8\varrho$, and by Li \cite{Li:BeurlingAhlfors}, who proved that $K_h\leq 4.2\varrho$. Later Lehtinen \cite{Lehtinen:BeurlingAhlfors} improved this bound to $2\varrho$, which is currently the best known bound; see also \cite{Lehtinen:BeurlingAhlfors2} and \cite{DeLin:BeurlingAhlfors}.

It is crucial for all these results to assume that $\rho_h(x,t) \leq \varrho$ \textit{for all} $x\in \R$, $t>0$, and one \textit{cannot} obtain in general any bound of the form 
\begin{align}\label{bound:dream}
K_h(x,y)\leq C \rho_h(x,y),
\end{align}
which would be an ideal bound for the extension problem. This was observed by Z.\ Chen and He \cite{Chen:BeurlingAhlfors,ChenHe:BeurlingAhlfors}, who gave examples of homeomorphisms $h$ such that $K_h(0,y)\rho_h(0,y)^{-1} \to\infty$ as $y\to0$. Z.\ Chen \cite{Chen:BeurlingAhlfors} also established, under no further assumptions on $h$, a bound of the form 
$$K_h(x,y)\leq C\rho_h(x,y)( \rho_h(x+y/2,y/2)+\rho_h(x-y/2,y/2) ),$$
which is, roughly speaking, of the form  $K_h=O(\rho_h^2)$. Nevertheless, this is a weak bound and does not imply sufficient integrability of $K_h$ for practical purposes. Therefore, in previous works, in order to obtain favorable bounds for $K_h$ in the spirit of \eqref{bound:dream}, further assumptions were imposed on the symmetric distortion function $\rho_h$. 

For instance, J.\ Chen, Z.\ Chen, and He \cite{ChenChenHe:boundary} proved that there exists a uniform constant $C>0$ such that if $\rho_h(x,t) \leq \varrho(t)$ for some decreasing function $\varrho(t)$, then 
\begin{align*}
K_h(x,y)\leq C \varrho(y/2)
\end{align*}
for all $x\in \R$, $y>0$. In fact, they claim this inequality with $\varrho(y)$ in place of $\varrho(y/2)$, but their proof contains an error that we point out in Remark \ref{remark:error} in the end of Section \ref{section:proof_main}. Zakeri in \cite{Zakeri:boundary} observed that this inequality can be used to obtain a David extension under the condition 
$$\varrho(y)= O(\log(1/y))$$
as $y\to 0$; the error mentioned above does not affect this result. We remark that Theorems \ref{theorem:david_line}--\ref{theorem:david_circle} and Corollary \ref{corollary:david_circle} are stronger.  Moreover, under the assumptions that $h(x+1)=h(x)+1$ for $x\in \R$ and $\exp(\sup_{y>0} \rho_h(\cdot ,y))\in L^q([0,1])$ for some $q>0$, Zakeri obtains a David extension by proving the bound
$$K_h(x,y)\leq 4 \max\left \{\sup_{y>0} \rho_h(x,y),  C_1(q)\log\left(\frac{C_2(h)}{y}\right)\right\},$$
where the constant $C_2(h)$ depends on the $L^q$ norm of $\exp(\sup_{y>0} \rho_h(x,y))$. De Faria \cite{deFaria:David} remarked that this inequality is not optimal, by constructing homeomorphisms $h$ of $\R$ that extend to David homeomorphisms of $\UHP$, but $\sup_{y>0} \rho_h(x,y) =\infty$ for a.e.\ $x\in \R$. His examples, though, satisfy the sufficient condition $\rho_h(x,y)= O(\log(1/y))$.

Under \textit{no assumptions} whatsoever on the homeomorphism $h$ we prove, as our main theorem, the following optimal bounds for the quasiconformal dilatation $K_h$ of the Beurling--Ahlfors extension. These bounds are close to the ideal bound \eqref{bound:dream}, but as we will see, for many practical purposes they are as good as that.

\begin{theorem}\label{theorem:main}
There exists a uniform constant $C_0>0$ such that the following holds. Let $h\colon \R\to \R$ be an increasing homeomorphism. Then 
\begin{align*}
\frac{\rho_h(x,y)}{4}\leq K_h(x,y) \leq  C_0 \max \left\{\rho_h(x,y), \frac{2}{y}\int_{-y/4}^{y/4} \rho_h(x+z,y-|z|)\, dz \right\}
\end{align*}
for all $x\in \R$ and $y>0$. 
\end{theorem}
This result implies the bounds of Chen et al.\ and Zakeri with possibly different constants. One can take $C_0=50$, but we have not attempted to optimize the value of the constant $C_0$. The integral in the right-hand side is the average of $\rho_h$ on the two segments, from the point $(x,y)$ to $(x-y/4,3y/4)$ and from $(x,y)$ to $(x+y/4,3y/4)$, as shown in Figure \ref{figure:average}. Moreover, upon integration, this result implies that $K_h$ and $\rho_h$ satisfy essentially the same integrability conditions.

\begin{figure}
\begin{tikzpicture}
\fill[gray!20!white](-5,0) rectangle (5,5);

\draw[->] (-5,0)-- (5,0);
\node[anchor=south] at (4,0) {$\R$};
\node[anchor=north] at (4,5) {$\UHP$};

\draw (-1,3)--(0,4)--(1,3);

\node at (0,4) {\textbullet};
\node[anchor=south] at (0,4) {$(x,y)$};

\node at (-1,3) {\textbullet};
\node[anchor=east] at (-1,3) {$(x-y/4,3y/4)$};

\node at (1,3) {\textbullet};
\node[anchor=west] at (1,3) {$(x+y/4,3y/4)$};
\end{tikzpicture}
\caption{The segments on which $\rho_h$ is averaged in Theorem \ref{theorem:main}.}\label{figure:average}
\end{figure}
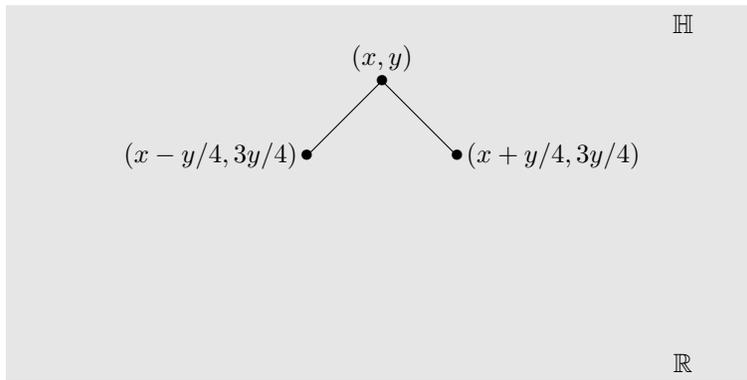

\begin{theorem}\label{theorem:convex_line}
Let $\Phi\colon [0,\infty)\to [0,\infty)$ be an increasing convex function and suppose that $h\colon \R\to \R$ is an increasing homeomorphism. Then for all $q>0$ we have
\begin{align*}
\int_{\UHP} \Phi(q4^{-1}C_0^{-1}\rho_h)  \, d\sigma\leq  \int_{\UHP} \Phi(qC_0^{-1}K_h) \, d\sigma \leq C\int_{\UHP} \Phi(q \rho_h)  \, d\sigma,
\end{align*}
where $C_0$ is the constant from Theorem \ref{theorem:main} and $C>0$ is a uniform constant.
\end{theorem}
Note that  Theorem \ref{theorem:david_line} follows from Theorem \ref{theorem:convex_line}. We prove Theorem \ref{theorem:main} in Section \ref{section:proof_main} and Theorem \ref{theorem:convex_line} in Section \ref{section:consequences}.

Another class of well-studied generalizations of quasiconformal maps are \textit{$BMO$-quasiconformal} maps. A homeomorphism $H\colon U\to V$ is $BMO$-quasiconformal if $H$ is a mapping of finite distortion and there exists $Q\in BMO(U)$ such that
$$K_H(x,y)\leq Q(x,y)$$
a.e.\ in $U$. In other words, the quasiconformal dilatation of $H$ has a $BMO$ \textit{majorant} (in $U$). This condition is locally slightly stronger than the condition of exponentially integrable distortion. These mappings were studied by Ryazanov, Srebro, and Yakubov \cite{RyazanovSrebroYakubov:BMO}. 

Sastry \cite[Theorem 3.1]{Sastry:boundaryBMO} established a sufficient condition for a homeomorphism $h\colon \R\to \R$ to admit a $BMO$-quasiconformal extension to $\UHP$ and Zakeri \cite[Theorem C]{Zakeri:boundary} proved a stronger result for homeomorphisms of $\R$ that commute with $x\mapsto x+1$. Namely, he proved that if $h$ commutes with $x\mapsto x+1$ and 
$$\rho_h(x,y) \leq \frac{1}{2y}\int_{x-y}^{x+y} A(t)\, dt,$$
where $A\in BMO(\R)$ and $A$ is $1$-periodic, then $h$ has a $BMO$-quasiconformal extension to $\UHP$.  His proof is based on the John--Nirenberg inequality \cite{JohnNirenberg:bmo} and the deep result of Bennett, DeVore, and Sharpley \cite{BennetDeVoreSharpley:bmo}, that the maximal function of a $BMO$ function is either identically equal to infinity or it lies in $BMO$.  {Using the estimate of the main Theorem \ref{theorem:main}, we prove an even stronger result with elementary means.}

For a function $A\in L^1_{\loc}(\UHP)$ and for $z\in \UHP$ we define 
$$\widehat{A}(z) = \frac{1}{|B_z|} \int_{B_z} A,$$
where $B_z$ is the ball $B(z, \im(z)/2)$ and $|\cdot |$ denotes the Lebesgue measure. We also define 
$$\widecheck{A}(z)= \frac{1}{|Q_z|}\int_{Q_z} A= \frac{1}{2y^2} \int_{y/2}^{3y/2}\int_{x-y}^{x+y} A,$$
where $Q_z$ is the $2y \times y$ rectangle, centered at $z=(x,y)$. We prove in Section \ref{section:bmo} the elementary fact that  $\widehat{\cdot}$ and $\widecheck{\cdot}$ are bounded operators from $BMO$ into itself. The combination of this fact with the main Theorem \ref{theorem:main} leads to the following result, proved in Section \ref{section:bmo}. Here, $\|A\|_*$ denotes the $BMO$ semi-norm of $A$.

\begin{theorem}\label{theorem:bmo}
There exists a uniform constant $C>0$ such that the following holds. Let $h\colon \R \to \R$ be an increasing homeomorphism and $A\in BMO(\UHP)$. 
\begin{align*}
&\textrm{If} \quad \rho_h\leq \widehat{A} \textrm{\,\, in $\UHP$} \quad \textrm{then}\quad K_h \leq C\widehat{A} +C\|A\|_* \textrm{\,\, in $\UHP$,}\quad \textrm{and}\\
&\textrm{if} \quad  K_h\leq \widehat{A} \textrm{\,\, in $\UHP$} \quad \textrm{then}\quad \rho_h \leq 4\widehat{A} \textrm{\,\, in $\UHP$}.
\end{align*}
The same conclusions hold with $\widecheck A$ in place of $\widehat{A}$. In particular, under any of these conditions, the Beurling--Ahlfors extension of $h$ to the upper half-plane is $BMO$-quasiconformal. 
\end{theorem}

We pose some questions for further study. It is proved in \cite[Proposition 2.5]{LyubichMerenkovMukherjeeNtalampekos:David} that David homeomorphisms of the unit disk are invariant under composition with quasiconformal homeomorphisms of the disk. Since the boundary maps of quasiconformal homeomorphisms of the disk are precisely quasisymmetric homeomorphisms of $S^1$, it follows that the circle homeomorphisms that have a David extension in the disk are invariant under composition with quasisymmetric maps. We pose, therefore, the following question.
\begin{question}
Let $h\colon S^1\to S^1$ be an orientation-preserving homeomorphism such that $e^{\rho_h^c}\in L^q([0,2\pi]\times [0,\pi/2])$ for some $q>0$. Is is true that the pre- and post-compositions of $h$ with quasisymmetric homeomorphisms of $S^1$ also have the same property (with a possibly different $q$)?
\end{question}
If the answer to Question \ref{question:necessary} is positive, then the answer to this question would also be positive.

Another natural problem is to characterize {welding homeomorphisms} of \textit{David circles}, i.e., Jordan curves that arise as the image of the unit circle under a global David homeomorphism of $\widehat{\C}$. A \textit{welding homeomorphism} is a homeomorphism of the circle that arises as the composition of the conformal map from the unit disk onto the interior region of a Jordan curve with a conformal map from the  exterior of this Jordan curve onto the exterior of the unit disk. The existence of a David extension of a circle homeomorphism $h$ inside the disk, as in the conclusion of Theorem \ref{theorem:david_circle}, implies that $h$ is a welding homeomorphism of a David circle. This follows from standard arguments; see, for instance, the discussion in \cite[Section 5]{LyubichMerenkovMukherjeeNtalampekos:David}. 

\begin{question}
What is a characterization of welding homeomorphisms of David circles?
\end{question}
For quasicircles, i.e., images of the unit circle under global quasiconformal maps, the characterization is known. Namely, quasisymmetric maps of the circle are precisely the welding homeomorphisms of quasicircles. However, we do not expect that the answer to the above question is the exponential integrability of $\rho_h^c$. The reason is that the inverse of a welding homeomorphism is a welding homeomorphism trivially. However, the inverse of a David homeomorphism is not necessarily a David map and likewise, we do not expect that the exponential integrability of $\rho_h^c$ is equivalent to the exponential integrability of $\rho_{h^{-1}}^c$.

\bigskip

\section{Proof of the main theorem}\label{section:proof_main}

In this section we prove the main Theorem \ref{theorem:main}. We first recall some basic facts and collect some properties of the Beurling--Ahlfors extension in Section \ref{section:beurling_ahlfors}, and then we give the proof of the theorem in Section \ref{section:proof_sub}. Our proof is self-contained for the convenience of the reader.

\subsection{The Beurling--Ahlfors extension}\label{section:beurling_ahlfors}

We recall the definition of the extension operator of Beurling and Ahlfors \cite{BeurlingAhlfors:extension}. Let $h \colon \R \to \R$ be an increasing homeomorphism. The Beurling--Ahlfors extension $H:\overline{\UHP} \to \overline{\UHP}$  of $h$ is defined by 
$$H(x,y)=u(x,y)+iv(x,y),$$
where
$$u(x,y)=\frac{1}{2y}\int_{x-y}^{x+y} h(t)\,dt \quad \textrm{and}\quad v(x,y)=\frac{1}{2y}\left( \int_{x}^{x+y} h(t)\,dt - \int_{x-y}^{x} h(t)\,dt\right)$$
for $x\in \R$ and $y>0$. Moreover, we define $H|_{\R}=h$. Beurling and Ahlfors proved in \cite[p.~135]{BeurlingAhlfors:extension} that $H:\overline{\UHP} \to \overline{\UHP}$ is a homeomorphism. Indeed, $H$ is proper, continuous, and locally injective and thus it is a covering map from  $\UHP$ onto itself. Since $\UHP$ is simply connected, it follows that $H:\UHP \to \UHP$ is a homeomorphism. This in conjunction with the fact that $H|_{\R}=h$ implies that $H$ is continuous and bijective on $\overline{\UHP}$. But $H^{-1}$ is also continuous on $\overline{\UHP}$ and hence $H:\overline{\UHP} \to \overline{\UHP}$ is a homeomorphism.

By general properties (see e.g.\ \cite[(21.1), p.~587]{AstalaIwaniecMartin:quasiconformal}), the quasiconformal dilatation $K_h$ of $H$ satisfies
\begin{align}\label{dilatation}
K_h+\frac{1}{K_h}=\frac{u_x^2+u_y^2+v_x^2+v_y^2}{u_xv_y-u_yv_x},
\end{align}
where in our case
\begin{align*}
u_x(x,y)&=\frac{1}{2y}\left(h(x+y)-h(x-y)\right),\\
u_y(x,y)&=\frac{1}{2y}\left( h(x+y)+h(x-y)-\frac{1}{y}\int_x^{x+y}h(t)\, dt-\frac{1}{y}\int_{x-y}^x h(t)\,dt\right),\\
v_x(x,y)&=\frac{1}{2y}\left(h(x+y)+h(x-y)-2h(x)\right), \quad \textrm{and}\\
v_y(x,y)&=\frac{1}{2y}\left( h(x+y)-h(x-y)-\frac{1}{y}\int_x^{x+y}h(t)\,dt+\frac{1}{y}\int_{x-y}^x h(t)\,dt\right).
\end{align*}

We list some transformation properties of the Beurling--Ahlfors extension. We let $a>0$ and $b\in \R$. The extension of a function $h$ is denoted by $H$ and the extension of a function $h^*$ is denoted by $H^*$.
\begin{enumerate}[label={(BA\arabic*)}]
	\item\label{p1} If $h^*(t)=ah(t)+b$, then $\rho_{h^*}(x,t)=\rho_h(x,t)$, $H^*(x,y)=aH(x,y)+b$, and $K_{h^*}(x,y)=K_h(x,y)$.
	\item\label{p2} If $h^*(t)=h(at+b)$, then $\rho_{h^*}(x,t)=\rho_h(ax+b,at)$, $H^*(x,y)=H(ax+b,ay)$, and $K_{h^*}(x,y)=K_h(ax+b,ay)$.
	\item\label{p3} If $h^*(t)=-h(-t)$, then $\rho_{h^*}(x,t)=\rho_h(-x,t)$, $H^*(x,y)=-\br{H(-x,y)}$, and $K_{h^*}(x,y)=K_h(-x,y)$; here $\br{H(-x,y)}$ denotes the complex conjugate of $H(-x,y)$.
\end{enumerate}
In all properties, the transformation of the extension $H$ follows immediately from the definition of the Beurling--Ahlfors extension. Moreover, the transformation of the symmetric distortion function $\rho_h$ is also immediate from the definition. The transformation of the quasiconformal dilatation follows only from the transformation of $H$ and does not depend on the properties of the Beurling--Ahlfors extension.

If $h$ is a normalized homeomorphism with $h(0)=0$ and $h(1)=1$, then
\[u_x(0,1)=\frac{1}{2}(1-h(-1)),\quad u_y(0,1)=\frac{1}{2}\left(1+h(-1)-\int_0^1 h(t)\,dt-\int_{-1}^0 h(t)\,dt\right)\]
and
\[v_x(0,1)=\frac{1}{2}(1+h(-1)),\quad v_y(0,1)=\frac{1}{2}\left(1-h(-1)-\int_0^1 h(t)\,dt+\int_{-1}^0 h(t)\,dt\right).\]
If we set 
\[\beta=-h(-1),\quad \xi=1-\int_0^1 h(t)\,dt,\quad \textrm{and}\quad \eta=1+\frac{1}{\beta}\int_{-1}^0 h(t)\,dt,\]
then $u_x(0,1)=\frac{1}{2}(1+\beta)$, $u_y(0,1)=\frac{1}{2}(\xi-\beta\eta)$, $v_x(0,1)=\frac{1}{2}(1-\beta)$, and $v_y(0,1)=\frac{1}{2}(\xi+\beta\eta)$.
Therefore, by \eqref{dilatation} we derive that
\begin{align}\label{dilatationxieta}
K_{h}(0,1)+\frac{1}{K_{h}(0,1)}=\frac{1}{\xi+\eta}\left(\beta(1+\eta^2)+\frac{1}{\beta}(1+\xi^2)\right).
\end{align}
By our normalization on $h$, it is clear that $\xi,\eta\in (0,1)$.

We denote by $F(\xi,\eta)$ the right-hand side of \eqref{dilatationxieta}, where $\beta$ is treated as a constant. Beurling and Ahlfors \cite[pp.~137--138]{BeurlingAhlfors:extension} observed that $F(\xi,\eta)$  is a convex function for $\xi,\eta>0$. Indeed, we set
\[F_1(\xi,\eta)=\frac{1+\eta^2}{\xi+\eta},\,\,F_2(\xi,\eta)=\frac{1+\xi^2}{\xi+\eta},\]
and observe that $F(\xi,\eta)=\beta F_1(\xi,\eta)+(1/\beta) F_2(\xi,\eta)$. Hence, it suffices to see that $F_1,F_2$ are convex. A direct calculation shows that	
\begin{equation}\nonumber 
	\mathrm{Hess} (F_1)= \mathrm{Hess} (F_2) =2
	\begin{bmatrix}
	(1+\eta^2)/(\xi+\eta)^3  & (1-\eta \xi)/(\xi+\eta)^3 \\\\
	(1-\eta \xi)/(\xi+\eta)^3 & (1+\xi^2)/(\xi+\eta)^3
	\end{bmatrix}.
	\end{equation}
The determinant is equal to $4/(\xi+\eta)^4>0$ and since $2(1+\eta^2)/(\xi+\eta)^3>0$ for $\xi,\eta>0$, we conclude that $F_1$ and $F_2$ are convex; see \cite[Theorem 4.5]{Rockafellar:convex}.

\bigskip

\subsection{Proof of Theorem \ref{theorem:main}}\label{section:proof_sub}

Throughout the proof we fix $x \in \R$ and $y>0$. Consider the normalized self-homeomorphism of $\R$ 
\[h^*(t)=\frac{h(x+yt)-h(x)}{h(x+y)-h(x)}\]
with $h^*(0)=0$ and $h^*(1)=1$. By properties \ref{p1} and \ref{p2}, it follows that 
\begin{align}\nonumber
\rho_{h^*}(s,t)=\rho_h (x+ys,yt)\quad \textrm{and}\quad K_{h^*} (s,t)=K_h (x+ys,yt)
\end{align}
for $s\in \R$ and $t>0$. For $s=0$ and $t=1$, by \eqref{dilatationxieta} we have
\begin{align}\label{main:dilatation_xi_eta}
K_{h}(x,y)+\frac{1}{K_{h}(x,y)}=\frac{1}{\xi+\eta}\left(\beta(1+\eta^2)+\frac{1}{\beta}(1+\xi^2)\right),
\end{align}
where 
\begin{align}\label{main:betaxieta}
\beta=-h^*(-1),\quad\xi=1-\int_0^1 h^*(t)\,dt,\quad \textrm{and}\quad\eta=1+\frac{1}{\beta}\int_{-1}^0 h^*(t)\,dt.
\end{align}

We first establish the lower estimate in the statement of the theorem.  Since $0<\xi,\eta<1$ and $K_h\geq 1$, by \eqref{main:dilatation_xi_eta} we have 
$$2K_{h}(x,y)\ge K_{h}(x,y)+\frac{1}{K_{h}(x,y)}\geq \frac{1}{2}\left(\beta+\frac{1}{\beta}\right).$$
Note that $\beta$ is equal to either $\rho_{h^*} (0,1)=\rho_h (x,y)$ or $1/\rho_{h^*} (0,1)=1/\rho_h (x,y)$. Thus, 
$$\frac{1}{2}\left(\beta+\frac{1}{\beta}\right)= \frac{1}{2}\left(\rho_h (x,y)+\frac{1}{\rho_h (x,y)}\right)\ge \frac{{\rho_h (x,y)}}{2}.$$
Therefore,
\begin{align*}
K_{h}(x,y)\ge \frac{{\rho_h (x,y)}}{4}.
\end{align*}
This completes the proof of the lower estimate.

\bigskip

In order to prove the upper estimate for $K_{h}(x,y)$, by \eqref{main:dilatation_xi_eta}, it suffices to prove the required estimate for the function
\[F(\xi,\eta)=\frac{1}{\xi+\eta}\left(\beta(1+\eta^2)+\frac{1}{\beta}(1+\xi^2)\right),\]
which dominates $K_h$. We consider two main cases: $\beta\geq 1$ and $\beta<1$. 

\begin{case}{1}\label{case1}
Suppose that $\beta \ge 1$. In this case, we have $\beta=\rho_h (x,y)$. To simplify the proof we take two further subcases. 
\end{case}

\begin{case}{1(a)}\label{case1a}
Suppose that
\begin{align}\label{main:case1}
h^*(-1/2)\le -\beta /2=h^*(-1)/2.
\end{align}
Essentially, this is the main non-trivial case and we will treat it in full detail. The remaining cases are either trivial or symmetric to this one. We split the proof in several steps for the convenience of the reader.
\end{case}

\noindent
\textit{Step 1:} We will show that 
\begin{align}\nonumber
\int_0^{1/4} h^*(t)\,dt \le \frac{1}{4}+\frac{h^*(-1/2)}{4}\frac{1}{1+A^*},
\end{align}
where 
\begin{align*}
A^*&=4\int_0^{1/4} \rho_{h^*}(t,1-t)\,dt=4\int_0^{1/4} \rho_{h}(x+ty,(1-t)y)\,dt\\
&=\frac{4}{y}\int_0^{y/4} \rho_{h}(x+z,y-z)dz.
\end{align*}

For $0\le t \le 1/4$, we have $2t-1 \le -1/2$, and since $h^*$ is increasing, it follows that $h^*(2t-1)\le h^*(-1/2)$. By this inequality and the definition of  $\rho_{h^*} (t,1-t)$, we infer that 
\begin{align}\nonumber
\frac{h^*(t)-h^*(-1/2)}{h^*(1)-h^*(t)}\le \frac{h^*(t)-h^*(2t-1)}{h^*(1)-h^*(t)} \le\rho_{h^*} (t,1-t)\eqqcolon  \rho^*
\end{align}
for $0\le t \le 1/4$.
Thus, $h^*(t)-h^*(-1/2)\le (1-h^*(t))\rho^*$, which implies that
\begin{align} \nonumber
h^*(t)\le \frac{\rho^*}{1+\rho^*}+\frac{h^*(-1/2)}{1+\rho^*}\le 1+\frac{h^*(-1/2)}{1+\rho^*}. 
\end{align}
Integrating over $t\in [0,1/4]$, we have
\begin{align*}
\int_0^{1/4} h^*(t)\,dt \le \frac{1}{4}+\frac{h^*(-1/2)}{4} \int_0^{1/4}\frac{4}{1+\rho^*}\,dt.
\end{align*}
Finally, by Jensen's inequality, 
\begin{align*}
\int_0^{1/4}\frac{4}{1+\rho^*}\,dt \ge \frac{1}{1+4\int_0^{1/4}\rho^* \,dt}=\frac{1}{1+A^*}.
\end{align*}
Since $h^*(-1/2)<0$, the desired conclusion follows.

\bigskip
\noindent
\textit{Step 2:} We will show that
\begin{align}\label{main:line}
4\xi + \frac{2}{3}\frac{\beta}{1+A^*}\eta \ge \frac{2}{3}\frac{\beta}{1+A^*}.
\end{align}
Geometrically, this inequality says that the point $(\xi,\eta)$ in the plane lies above a certain line with slope $-6(1+A^*)\beta^{-1}$; see Figure \ref{quadri}.

Using the estimate from Step 1, we obtain the estimates
\begin{align}
\begin{aligned}\label{main:4xi}
4\xi&=4\left(1-\int_0^1 h^*(t)\,dt \right)=4-4\left( \int_0^{1/4} h^*(t)\,dt+\int_{1/4}^1 h^*(t)\,dt\right)  \\
&\ge 4-4\left( \frac{1}{4}+\frac{h^*(-1/2)}{4}\frac{1}{1+A^*}+\frac{3}{4}\cdot 1 \right)=-\frac{h^*(-1/2)}{1+A^*}.
\end{aligned}
\end{align}
Moreover, by the main assumption \eqref{main:case1} of Case \ref{case1a}, we have
\begin{align*}
\frac{\beta}{1+A^*}\eta&=\frac{\beta}{1+A^*}\left(1+\frac{1}{\beta}\int_{-1}^0 h^*(t)\,dt \right) \\
&=\frac{\beta}{1+A^*}+\frac{1}{1+A^*}\left(\int_{-1}^{-1/2} h^*(t)\,dt+\int_{-1/2}^0 h^*(t)\,dt\right) \\
&\ge \frac{\beta}{1+A^*}+\frac{1}{1+A^*} \frac{h^*(-1)+h^*(-1/2)}{2}  \\
&\ge \frac{\beta}{1+A^*}+\frac{3}{2} \frac{h^*(-1/2)}{1+A^*}.
\end{align*}
Equivalently, 
$$\frac{2}{3}\frac{\beta}{1+A^*}\eta \geq \frac{2}{3}\frac{\beta}{1+A^*}+ \frac{h^*(-1/2)}{1+A^*}$$
Adding this inequality to \eqref{main:4xi} leads to the claimed inequality.

\bigskip
\noindent
\textit{Step 3:} We finally estimate $F(\xi,\eta)$ from above, under the restrictions $0<\xi,\eta<1$ and under inequality  \eqref{main:line} from Step 2. We will consider two cases, depending on whether or not $\frac{\beta}{6(1+A^*)}$ is less than $1$.

Suppose first that
\[\frac{\beta}{6(1+A^*)}<1.\]
We deduce that the point $(\xi,\eta)$ lies in the convex quadrilateral $Q$ (see Figure \ref{quadri}) bounded by the lines 
\[\xi=1,\quad \eta=0,\quad \eta=1,\quad \textrm{and}\quad  4\xi + \frac{2}{3}\frac{\beta}{1+A^*}\eta =\frac{2}{3}\frac{\beta}{1+A^*}.\]
Since $F$ is convex, its maximum in $Q$ is attained at one of the vertices $(0,1)$, $(1,1)$, $(1,0)$, and $(\beta /{(6(1+A^*))},0)$; see \cite[Corollary 32.3.4]{Rockafellar:convex}. Estimating the values of $F$ at these vertices, we have 
\begin{align*}
F(0,1)&=2\beta+\frac{1}{\beta}\le 2\beta+1 \le 3\beta= 3\rho_h(x,y),\\
F(1,1)&=\beta+\frac{1}{\beta}\le F(0,1),\\
F(1,0)&=\beta+\frac{2}{\beta}\le F(0,1), \quad \textrm{and}\quad \\
F\left(\frac{\beta}{6(1+A^*)},0\right)&=6(1+A^*)\left(1+\frac{1}{\beta^2}\right)+\frac{1}{6(1+A^*)}\\
&\le 12(1+A^*)+\frac{1}{12}\le 25A^*, 
\end{align*}
because $\beta,A^*\ge1$. Thus,
\begin{align}\label{main:maxH}
F(\xi,\eta)\le 25 \max \left\{\rho_h(x,y), A^*\right\}.
\end{align}

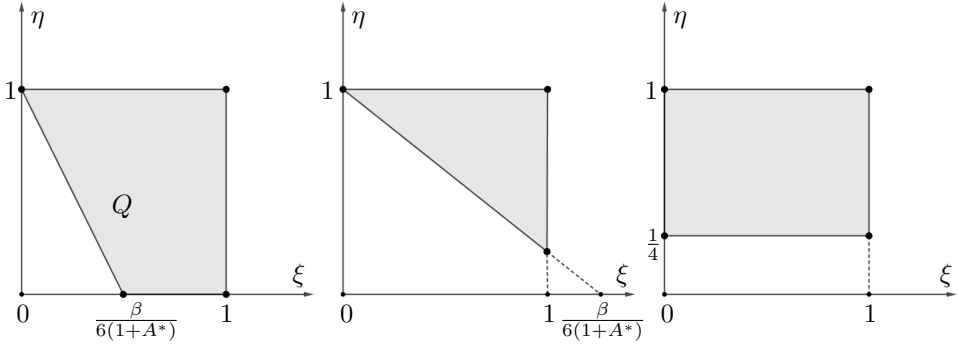
\begin{figure}
		\begin{overpic}[width=\linewidth]{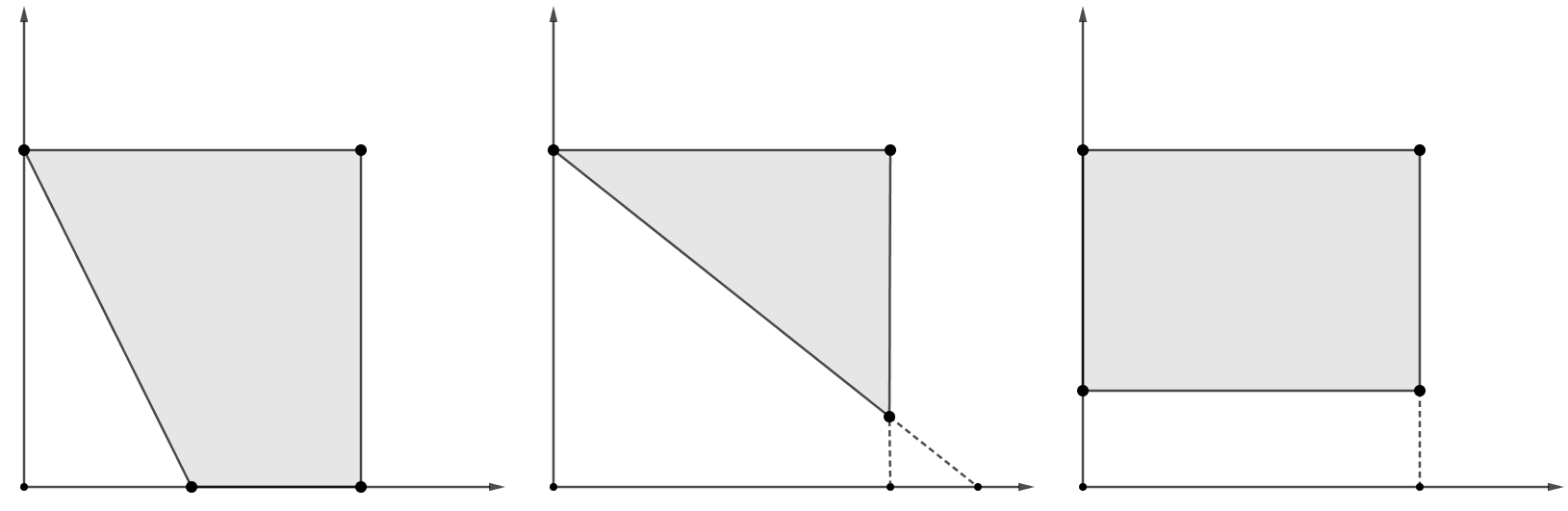}
		\put (11,10) {$Q$}
		\put (1,-1.5) {$0$}
		\put (35,-1.5) {$0$}
		\put (68.7,-1.5) {$0$}
		\put (22.4,-1.5) {$1$}
		\put (56.3,-1.5) {$1$}
		\put (90,-1.5) {$1$}
		\put (9,-2) {$\frac{\beta}{6(1+A^*)}$}
		\put (58,-2) {$\frac{\beta}{6(1+A^*)}$}
		\put (-0.3,22) {$1$}
		\put (33,22) {$1$}
		\put (67,22) {$1$}
		\put (67,6) {$\frac{1}{4}$}
		\put (2.5,30) {$\eta$}
		\put (36.2,30) {$\eta$}
		\put (70,30) {$\eta$}
		\put (30,2.7) {$\xi$}
		\put (64,2.7) {$\xi$}
		\put (98,2.7) {$\xi$}
		\end{overpic}
	\vspace*{0.2mm}
	\caption{Convex polygons containing the point $(\xi,\eta)$ in Case \ref{case1}.}
	\label{quadri}	
\end{figure}

If, instead, \[\frac{\beta}{6(1+A^*)}\ge 1,\]
then $(\xi,\eta)$ lies in the triangle (see Figure \ref{quadri}) bounded by the lines 
\[\xi=1,\quad \eta=1,\quad \textrm{and}\quad  4\xi + \frac{2}{3}\frac{\beta}{1+A^*}\eta =\frac{2}{3}\frac{\beta}{1+A^*},\]
which is contained in the triangle with vertices $(0,1)$, $(1,1)$, and $(1,0)$. The maximum of $F$ in this triangle is attained at one of the vertices. From the previous estimates, $F$ is bounded above by $3\rho_h(x,y)$. Thus, \eqref{main:maxH} also holds in this case.

\bigskip

\begin{case}{1(b)}\label{case1b}
Suppose that $h^*(-1/2)>-\beta /2$.
\end{case}

\noindent 
Then 
\begin{align}
\eta&= 1+ \frac{1}{\beta}\int_{-1}^0 h^*(t)\,dt=1+\frac{1}{\beta}\left(\int_{-1}^{-1/2} h^*(t)\,dt+\int_{-1/2}^{-1} h^*(t)\,dt\right) \nonumber \\
&\ge 1+\frac{1}{2\beta}(h^*(-1)+h^*(-1/2))\ge 1-\frac{\beta+\beta/2}{2\beta}=\frac{1}{4}. \nonumber
\end{align}
Since $0<\xi<1$ and $1/4\le \eta <1$, the point $(\xi, \eta)$ lies in the rectangle (see Figure \ref{quadri}) bounded by the lines
\[\xi=0,\quad \xi=1,\quad  \eta=1/4,\quad \textrm{and}\quad \eta=1.\]
Hence, the function $F$ reaches its maximum at one of the vertices $(0,1/4)$, $(1,1/4)$, $(1,1)$, and $(0,1)$. Since
\begin{align*}
F(0,1/4)&=\frac{17\beta}{4}+\frac{4}{\beta}\le \frac{17\beta}{4}+4\le 9\beta =9\rho_h (x,y),\quad \textrm{and}\\
F(1,1/4)&=\frac{17\beta}{20}+\frac{8}{5\beta}\le3\beta=3\rho_h (x,y),
\end{align*}
the relation \eqref{main:maxH} is still true. 

\bigskip

Thus, in both Cases \ref{case1a} and \ref{case1b}, we derive that 
\[F(\xi,\eta)\le 25 \max \left\{\rho_h(x,y), A^*\right\}.\]
This in conjunction with \eqref{main:dilatation_xi_eta} gives
\begin{align}\label{main:estimatecase1}
K_h(x,y)\le 25 \max \left\{ \rho_h (x,y), \frac{4}{y}\int_0^{y/4} \rho_{h}(x+z,y-z)\,dz \right\}. 
\end{align}

\bigskip

\begin{case}{2}\label{case2}
Suppose that $\beta<1$. This case is symmetric to Case \ref{case1}.
\end{case}

\noindent
In this case, we have $\beta=1/\rho_h(x,y)$. We consider the normalized increasing homeomorphism  ${\widetilde h} (t)=-h^*(-t)/\beta$ with $\widetilde h(0)=0$ and $\widetilde h(1)=1$. By properties \ref{p1} and \ref{p3}, it follows that
\[\rho_{\widetilde h} (s,t)= \rho_{h^*}(-s,t)=\rho_{h}(x-ys,yt)\quad \textrm{and}\quad K_{\widetilde{h}}(s,t)=K_{h^*}(-s,t)=K_{h}(x-ys,yt).\]
So, if in \eqref{main:dilatation_xi_eta} and \eqref{main:betaxieta} we replace  $\beta$, $\xi$, $\eta$, and $h^*$ by $\widetilde \beta$, $\widetilde \xi$, $\widetilde \eta$, and $\widetilde{h}$, respectively, then we have 
\begin{align*}
K_h(x,y)= K_{\widetilde h}(0,1)  \leq \frac{1}{\widetilde \xi+\widetilde \eta}\left(\widetilde \beta(1+\widetilde{\eta\mkern 0mu}^{2})+\frac{1}{\widetilde \beta}(1+\widetilde{\xi\mkern 0mu}^{2})\right) \eqqcolon \widetilde F(\widetilde \xi, \widetilde \eta)
\end{align*}
and 
$$\widetilde \beta=-\widetilde{h} (-1)=\frac{1}{\beta}=\rho_h(x,y)>1.$$
This reduces Case \ref{case2} to Case \ref{case1}. Hence, we obtain the conclusion
$$\widetilde F(\widetilde \xi, \widetilde \eta) \leq 25 \max\{\rho_h(x,y), \widetilde A\},$$
where
\begin{align}
\widetilde{A}&=4\int_0^{1/4} \rho_{\widetilde{h}}(t,1-t)\,dt=4\int_0^{1/4} \rho_{h}(x-ty,(1-t)y)\,dt \nonumber\\
&=\frac{4}{y}\int_{-y/4}^0 \rho_{h}(x+z,y+z)\,dz \nonumber.
\end{align}
We deduce that
\begin{align}\label{main:estimatecase2}
K_h(x,y)\le 25 \max \left\{ \rho_h (x,y), \frac{4}{y}\int_{-y/4}^0 \rho_{h}(x+z,y+z)\,dz \right\}. 
\end{align}

\bigskip

Combining \eqref{main:estimatecase1} from Case \ref{case1} and \eqref{main:estimatecase2} from Case \ref{case2}, we finally have
\[K_h(x,y)\le 50 \max \left\{ \rho_h (x,y), \frac{2}{y}\int_{-y/4}^{y/4} \rho_{h}(x+z,y-|z|)\,dz \right\}.\]
This completes the proof.\qed

\begin{remark}\label{remark:error}
Our proof above follows some ideas from the proof of Theorem 3 in \cite{ChenChenHe:boundary}. However, there is an error in that proof. Namely, the last inequality in (3.10) is incorrect, since $h(2t-1)$ is a negative number for $t\in (0,1/2)$. Since $\rho$ is assumed to be a decreasing function, in order to obtain a correct estimate one would have to replace $\rho(y_0)$ by $\rho(y_0/2)$ in the last displayed formula of (3.10). This alters the conclusion of the theorem, inequality (3.4), to the inequality
$D(x_0+iy_0)\leq 4\rho(y_0/2)+C.$
\end{remark}

\bigskip

\section{Consequences of the main theorem}\label{section:consequences}

In this section we establish the consequences of the main theorem; that is, Theorem \ref{theorem:convex_line}, which provides integrability conditions for homeomorphisms of the real line, Theorem \ref{theorem:convex_circle}, which provides integrability conditions for homeomorphisms of the circle, and Theorem \ref{theorem:bmo}, regarding the $BMO$ majorants.

\subsection{Integrability conditions}

\begin{proof}[Proof of Theorem \ref{theorem:convex_line}]

Let $\Phi\colon [0,\infty)\to [0,\infty)$ be an increasing convex function. The first inequality of Theorem \ref{theorem:main} immediately implies the first inequality of Theorem \ref{theorem:convex_line}.

For the second inequality, we multiply the second inequality of Theorem \ref{theorem:main} with $qC_0^{-1}$, and then apply the function $\Phi$. Using Jensen's inequality we obtain
\begin{align*}
\Phi(qC_0^{-1}K_h(x,y)) \leq  \max \left\{\Phi(q\rho_h(x,y)), \frac{2}{y}\int_{-y/4}^{y/4} \Phi(q\rho_h(x+z,y-|z|))\, dz \right\},
\end{align*}
for all $(x,y)\in \UHP$. In order to obtain the conclusion, it suffices to show that
\begin{align*}
\int_{\R} \int_{0}^{\infty} \frac{2}{y}\int_{-y/4}^{y/4} \Phi(q\rho_h(x+z,y-|z|)) \, dz \, \frac{4dydx}{(1+x^2+y^2)^2} \leq C  \int_{\UHP} \Phi(q\rho_h) \, d\sigma
\end{align*}
for some uniform constant $C>0$. For simplicity, we define $f=\Phi(q\rho_h)$ and let $g(x,y)$ be the spherical density $4/(1+x^2+y^2)^2$. 

We break the inner integral into two integrals, from $0$ to $y/4$ and from $-y/4$ to $0$. It suffices to prove the estimate for the first one, since the proof is essentially the same for the other. By changing coordinates repeatedly and applying Fubini's theorem, we have
\begin{align*}
\int_{\R} \int_{0}^{\infty} \frac{2}{y}\int_0^{y/4} &f(x+z,y-z)g(x,y) \, dzdydx\\
&\quad\quad= \int_{\R} \int_{0}^{\infty} \frac{2}{y}\int_{3y/4}^{y} f(x+y-u,u)g(x,y) \, dudydx\\
&\quad\quad= \int_{0}^{\infty} \frac{2}{y}\int_{3y/4}^{y} \int_{\R}f(x+y-u,u)g(x,y) \, dxdudy\\
&\quad\quad= \int_{0}^{\infty} \frac{2}{y}\int_{3y/4}^{y} \int_{\R}f(w,u)g(w+u-y,y) \, dwdudy.
\end{align*}
Now we claim that $g(w+u-y,y)\leq 16 g(w,u)$ for $0<u<y$ and $w\in \R$. Indeed,
\begin{align*}
1+(w+u-y)^2+y^2&\geq 1+(w-(y-u))^2 + \frac{(y-u)^2+u^2}{2} \\
&\geq 1+\frac{(w-(y-u))^2 +(y-u)^2}{2} +\frac{u^2}{2}\\
&\geq 1+ \frac{w^2}{4}+\frac{u^2}{2} \geq \frac{1+w^2+u^2}{4}.
\end{align*}
The claim follows immediately. Therefore, it suffices to bound the integral of $f(w,u)g(w,u)$, instead of $f(w,u)g(w+u-y,y)$. We have
\begin{align*}
\int_{0}^{\infty} \frac{2}{y}\int_{3y/4}^{y} \int_{\R}f(w,u)g(w,u) \, dwdudy&= \int_{\R}\int_{0}^\infty  f(w,u)g(w,u)\int_{u}^{4u/3} \frac{2}{y}\,dydudw.
\end{align*}
Finally, we observe that 
$$\int_{u}^{4u/3} \frac{2}{y}\,dy= 2\log(4/3)$$
for all $u>0$ and this completes the proof.
\end{proof}

Next, we prove Theorem \ref{theorem:convex_circle}, which is a transportation of Theorem \ref{theorem:convex_line} to homeomorphisms of the circle. Our proof follows from an adaptation of the argument of Zakeri \cite[p.~243]{Zakeri:boundary}.

\begin{proof}[Proof of Theorem \ref{theorem:convex_circle}]
Let $h\colon S^1\to S^1$ be an orientation-preserving homeomorphism of the circle. If $h(1)= e^{i\theta_0}\neq 1$, we consider the homeomorphism $h\cdot e^{-i\theta_0}$. If we extend this homeomorphism to a homeomorphism $H$ of the disk with the desired integrability properties for $K_H$, then $H\cdot e^{i\theta_0}$ will be an extension of $h$ with the desired properties. Therefore, it suffices to prove the theorem assuming that $h(1)=1$.

We lift the homeomorphism $h$ to the real line, under the universal covering map $\psi(z)=e^{2\pi i z}$. We thus obtain an increasing homeomorphism $\widetilde h$ of the real line with $\widetilde h(0)=0$, $\widetilde h(1)=1$, and $\widetilde h(x+1)=\widetilde h(x)+1$ for all $x\in \R$. Consider the Beurling--Ahlfors extension $\widetilde H$ of $\widetilde h$ in the upper half-plane. By properties \ref{p1} and \ref{p2}, we have $\widetilde H(z+1)=\widetilde H(z)+1$ for all $z\in \UHP$. It follows that $\widetilde H$ descends to a homeomorphism $H$ of the unit disk that extends $h$.

Since the circular symmetric distortion $\rho_h^c$ is defined using arclength, we have $\rho_h^c(\theta,t) =\rho_{\widetilde h}(\theta/2\pi, t/2\pi)$ for all $\theta\in [0,2\pi]$ and $t\in (0,\pi/2]$. The assumption that $\Phi(q \rho_h^c)\in L^1([0,2\pi]\times(0,\pi/2])$ now implies that $\Phi(q\rho_{\widetilde h}) \in L^1([0,1]\times (0,1/4])$. The continuity of $\rho_{\widetilde h}$ on $[0,1]\times [1/4,1]$ implies that $\Phi(q\rho_{\widetilde h}) \in L^1([0,1]\times (0,1])$. Since $\widetilde h$ commutes with $x\mapsto x+1$, we have that $\rho_{\widetilde h}$ is bounded on $[0,1]\times [1,\infty)$. Therefore, $\Phi(q\rho_{\widetilde h})  \in L^1 ([0,1]\times (0,\infty); d\sigma)$. Again, since $\widetilde h$ commutes with $x\mapsto x+1$, we conclude that $\Phi(q\rho_{\widetilde h})  \in L^1( \UHP; d\sigma)$. 

By Theorem \ref{theorem:convex_line}, the Beurling--Ahlfors extension $\widetilde H$ of $\widetilde h$ satisfies $\Phi(qC_0^{-1}K_{\widetilde H})  \in L^1( \UHP; d\sigma)$. Since $\psi \circ \widetilde H= H\circ \psi$ and $\psi$ is locally conformal, we have $K_{\widetilde H}= K_H \circ \psi$. By changing coordinates under the conformal map $\psi|_{(0,1)\times (0,\infty)}$ we obtain
\begin{align}\label{equality:jacobian}
\int_{\D \setminus [0,1)\times \{0\}} \Phi(qC_0^{-1}K_{H}) J_{\psi^{-1}}^{\sigma}\, d\sigma =\int_{(0,1)\times (0,\infty)}  \Phi(qC_0^{-1}K_{\widetilde H}) \, d\sigma<\infty,
\end{align}
where $J_{\psi^{-1}}^{\sigma}$ denotes the spherical Jacobian  of $\psi^{-1}(z)= \frac{\log(z)}{2\pi i}$, with a branch cut along the non-negative real axis. We have
\begin{align*}
J_{\psi^{-1}}^{\sigma}(z)= J_{\psi^{-1}}(z) \frac{(1+|z|^2)^2}{(1+|\psi^{-1}(z)|^2)^2}\simeq \frac{1}{|z|^2}\frac{(1+|z|^2)^2}{(1+\log^2|z|)^2} \gtrsim 1
\end{align*}
for $z\in \D\setminus[0,1)\times \{0\}$. Since $d\sigma\simeq dxdy$ for $(x,y)\in \D$, by \eqref{equality:jacobian} we have
\begin{align*}
\int_{\D} \Phi(qC_0^{-1}K_{H})\, dxdy <\infty.
\end{align*}
This completes the proof.
\end{proof}

\bigskip

\subsection{Functions of bounded mean oscillation}\label{section:bmo}
We first recall the definition of a function of bounded mean oscillation. Let $U\subset \R^n$, $n\geq 1$, be an open set and $A\in L^1_{\loc}(U)$. The function $A$ lies in $BMO(U)$ if 
\begin{align*}
\|A\|_*\coloneqq \sup_{B\subset U} \frac{1}{|B|} \int_B |A-A_B| <\infty,
\end{align*}
where $A_B= \frac{1}{|B|}\int_B A$,  and the supremum is taken over all closed balls $B\subset U$.

For $A\in L^1_{\loc}(U)$, we define
$$\widehat{A}(z) = \frac{1}{|B_z|} \int_{B_z} A,$$
where $B_z=B(z,\dist(z,\partial U)/2)$. 
\begin{lemma}\label{lemma:bmo_average}
There exists a uniform constant $C=C(n)>0$ such that if $A\in BMO(U)$, then $\widehat{A}\in BMO(U)$ and $\|\widehat A\|_*\leq C\|A\|_*$.
\end{lemma}
\begin{proof}
Suppose that $A\in BMO(U)$ and consider a ball $B_0=B(z_0,r)\subset \br{B(z_0,r)}\subset U$. We will show that there exists a uniform constant $C=C(n)>0$ and  a constant $c_0\in \R$ depending on $B_0$ such that
\begin{align}\label{lemma:bmo_equivalent}
\int_{B_0} |\widehat{A}- c_0| \leq C\|A\|_*|B_0|.
\end{align}
This will imply that $\widehat{A}\in BMO(U)$ and $\|\widehat{A}\|_{*}\leq 2C\|A\|_*$; see \cite[Lemma 14.49, p.~445]{WheedenZygmund:real}.

Let $z\in B_0$ and consider a chain of points $z_0,z_1,\dots,z_N=z$ lying on the segment between $z_0$ and $z$ such that $|z_i-z_{i-1}|=2^{-i}r$ for $i\in \{1,\dots,N-1\}$ and $|z_N-z_{N-1}|\leq 2^{-N}r$. We fix $i\in \{1,\dots,N\}$. Recall that $B_{z_i}=B(z_i,d_i/2)$, where $d_i=\dist(z_i,\partial U)$. If $d_i\leq d_{i-1}$, then we define $$B(w_i,R_i)= B(z_{i-1}, d_{i-1}/2 +|z_{i}-z_{i-1}|),$$ while if $d_{i-1}<d_i$, then we define $$B(w_i,R_i)=B(z_{i}, d_{i}/2 +|z_{i}-z_{i-1}|).$$
In both cases we have
$$B_{z_i}\cup B_{z_{i-1}}\subset B(w_i,R_i).$$
We observe that 
$$d_{i-1}\geq d_0-|z_{i-1}-z_0|>r- r(2^{-1}+\dots+ 2^{-i+1})=2^{-i+1}r\geq  2|z_i-z_{i-1}|.$$
It follows that 
\begin{align}\label{lemma:bmo_Ri_def}
R_i=\max\{d_{i-1},d_i\}/2 +|z_{i}-z_{i-1}|<\max\{d_{i-1},d_i\}
\end{align}
and thus $B(w_i,R_i)\subset \br{B(w_i,R_i)}\subset U$.  Since $|d_i-d_{i-1}|\leq |z_i-z_{i-1}|<d_{i-1}/2$, we have $d_{i-1}< 2d_i$ and $d_{i}<3d_{i-1}/2$. Therefore, $\max\{d_{i-1},d_i\}<2\min\{d_{i-1},d_i\}$. This, in conjunction with \eqref{lemma:bmo_Ri_def}, gives
\begin{align}\label{lemma:bmo_Ri}
R_i/2 <\min\{d_i,d_{i-1}\} \leq \max\{d_i,d_{i-1}\} < 2R_i.
\end{align}

Since $B_{z_i}\subset B(w_i,R_i)$ and these balls have comparable radii, we have
\begin{align*}
|A_{B_{z_{i}}}-A_{B(w_i,R_i)}| &\leq \frac{1}{|B_{z_{i}}|} \int_{B_{z_{i}}} |A-A_{B(w_i,R_i)}| \\
&\leq \frac{C}{|B(w_i,R_i)|}\int_{B(w_i,R_i)} |A-A_{B(w_i,R_i)}|,
\end{align*}
for a uniform constant $C=C(n)>0$. The fact that the closure of $B(w_i,R_i)$ is contained in  $U$ and the assumption that $A\in BMO(U)$ imply that the latter average is bounded by $\|A\|_*$. Hence,
\begin{align*}
|A_{B_{z_{i}}}-A_{B(w_i,R_i)}|\leq C \|A\|_{*}\quad \textrm{and similarly}\quad|A_{B_{z_{i-1}}}- A_{B(w_i,R_i)}|\leq C\|A\|_*
\end{align*}
for a uniform constant $C=C(n)>0$. Therefore, we have
\begin{align*}
|\widehat{A}(z)-\widehat{A}(z_0)|&\leq \sum_{i=1}^N |\widehat{A}(z_i)-\widehat{A}(z_{i-1})| \\
&\leq \sum_{i=1}^N(|A_{B_{z_i}}- A_{B(w_i,R_i)}| + |A_{B_{z_{i-1}}}-A_{B(w_i,R_i)}|)\\
&\leq 2CN\|A\|_*.
\end{align*}
Note that $|z-z_0|\geq \sum_{i=1}^{N-1} |z_{i}-z_{i-1}| = r\sum_{i=1}^{N-1}2^{-i}= r(1-2^{-N+1})$. Therefore, 
$$N\leq 1+(\log2)^{-1} \log\left(\frac{1}{1-|z-z_0|/r} \right).$$ 
Finally, by integrating over $B_0$ we have
\begin{align*}
\int_{B_0} |\widehat{A}(z)-\widehat{A}(z_0)| \, dz \leq C'|B_0| \|A\|_* + C' \|A\|_* \int_{B_0} \log\left(\frac{1}{1-|z-z_0|/r} \right) \, dz,
\end{align*}
where $C'=C'(n)>0$ is a uniform constant. By integrating in polar coordinates, we see that the latter integral is bounded by $C''|B_0|$ for a constant $C''=C''(n)>0$. Hence, we have proved \eqref{lemma:bmo_equivalent} with $c_0=\widehat{A}(z_0)$.
\end{proof}

Recall that for $A\in L^1_{\loc}(\UHP)$ we have defined
$$\widecheck{A}(z)= \frac{1}{|Q_z|}\int_{Q_z} A= \frac{1}{2y^2} \int_{y/2}^{3y/2}\int_{x-y}^{x+y} A,$$
where $Q_z$ is the $2y \times y$ open rectangle, centered at $z=(x,y)$. 

\begin{lemma}\label{lemma:bmo_average_rectangle}
There exists a uniform constant $C>0$ such that if $A\in BMO(\UHP)$, then 
\begin{enumerate}[\upshape(i)]
\item $|\widecheck A-\widehat A|\leq C\|A\|_*$ in $\UHP$, and
\item $\widecheck{A}\in BMO(\UHP)$ with $\|\widecheck A\|_*\leq C\|A\|_*$.
\end{enumerate}
\end{lemma}
Of course, the particular choice of the dimensions of the rectangle $Q_z$ is not of importance, as long as the lengths of the sides of $Q_z$ are comparable to the distance of $Q_z$ to the boundary of $\UHP$.

\begin{proof}
For any $z=(x,y)\in \UHP$ there exists a ball $B_0\subset \br {B_0}\subset \UHP$ with radius comparable to $y$, such that $B_0\supset B_z=B(z,y/2)$ and $B_0\supset Q_z$. For example, one can take the center to be $(x,2y)$ and the radius to be $y\sqrt{13}/2$.  We now have
\begin{align*}
|\widecheck{A}(z) -\widehat{A}(z)| &\leq |\widecheck{A}(z)-A_{B_0}| +|A_{B_0}-\widehat A(z)| \\
&\leq \frac{1}{|Q_z|} \int_{Q_z} |A-A_{B_0}| + \frac{1}{|B_z|} \int_{B_z} |A-A_{B_0}|\\
&\leq \frac{C'}{|B_0|} \int_{B_0} |A-A_{B_0}| \leq C'\|A\|_*
\end{align*}
for a uniform constant $C'>0$, since $A\in BMO(\UHP)$. 

Upon integration, it follows that $|(\widecheck A)_B- (\widehat A)_B| \leq C'\|A\|_*$ for any ball $B\subset \br B\subset \UHP$. Therefore, by the above and Lemma \ref{lemma:bmo_average} we have
\begin{align*}
\frac{1}{|B|}\int_B |\widecheck A- (\widecheck A)_B| &\leq \frac{1}{|B|}\int_B |\widecheck A- \widehat{A}|+\frac{1}{|B|}\int_B |\widehat A- (\widehat A)_B|+\frac{1}{|B|}\int_B |(\widehat A)_B- (\widecheck A)_B|\\
&\leq C'\|A\|_* +C\|A\|_* +C'\|A\|_* =(2C'+C)\|A\|_*,
\end{align*}
where $C$ is the constant from Lemma \ref{lemma:bmo_average}. This completes the proof.
\end{proof}

\begin{proof}[Proof of Theorem \ref{theorem:bmo}]
If $K_h\leq \widehat{A}$, then $\rho_h\leq 4\widehat A$ by the first inequality of Theorem \ref{theorem:main}. The same claim holds with $\widecheck A$ in place of $\widehat A$.

Conversely, suppose that $\rho_h\leq \widehat{A}$. By the second inequality of Theorem \ref{theorem:main}, we have
$$K_h(x,y)\leq C_0\widehat{A}(x,y)+ C_0 \frac{2}{y}\int_{-y/4}^{y/4} \widehat{A}(x+z,y-|z|)\, dz, $$
Hence, it suffices to show that 
$$\frac{2}{y}\int_{-y/4}^{y/4} \widehat{A}(x+z,y-|z|)\, dz \leq \widehat{A}(x,y)+C\|A\|_*$$
for a uniform constant $C>0$. We fix $z\in (-y/4,y/4)$ and consider the ball $B=B((x,y),R)$, where $R=y\sqrt{2}/4+ y/2<y$, so that 
$$\br B\subset \UHP,\quad B((x,y),y/2) \subset B, \quad\textrm{and}\quad B( (x+z,y-|z|),(y-|z|)/2)\subset B.$$ We now estimate
\begin{align*}
|\widehat{A}(x+z,y-|z|)-\widehat{A}(x,y)| &\leq |\widehat{A}(x+z,y-|z|)-A_B|+|A_B-\widehat{A}(x,y)|\\
&\leq \frac{C}{|B|} \int_B |A-A_B|\leq C\|A\|_*
\end{align*}
for a uniform constant $C>0$. Therefore,
\begin{align*}
\widehat{A}(x+z,y-|z|) \leq \widehat{A}(x,y) + C\|A\|_*
\end{align*}
for all $z\in (-y/4,y/4)$. Upon integration, this completes the proof in this case.

If we have $\rho_h\leq \widecheck A$, then by Lemma \ref{lemma:bmo_average_rectangle} (i) we have 
$$\rho_h\leq \widehat{A}+C'\|A\|_* = \widehat{(A+C'\|A\|_*)}$$
for a uniform constant $C'>0$. By the previous case, we have
$$K_h \leq C'' ( \widehat{(A+C'\|A\|_*)}  +  \|A+C'\|A\|_* \|_* ) = C'' \widehat A + C''(C'+1)\|A\|_*$$
for some uniform constant $C''>0$. Applying again Lemma \ref{lemma:bmo_average_rectangle} (i) and switching back to $\widecheck A$ leads to the desired conclusion.
\end{proof}

\bigskip

\bibliography{biblio}
\end{document}